\documentclass[12pt]{amsart}
\usepackage{amsmath,amsthm,amssymb,amsxtra,enumerate,bbm,dsfont,hyperref}
\allowdisplaybreaks
\newtheorem{lemma}{Lemma}[section]
\newtheorem{proposition}[lemma]{Proposition}
\newtheorem{theorem}[lemma]{Theorem}
\newtheorem{corollary}[lemma]{Corollary}
\newtheorem{remark}[lemma]{Remark}
\newtheorem{definition}[lemma]{Definition}
\newtheorem{example}[lemma]{Example}
\newcommand\1{\mathds{1}} 
\newcommand{\one}{\1}
\newcommand{\PP}{{\mathbb{P}}}
\newcommand{\DD}{{\mathbb{D}}}
\newcommand{\ew}{\mathbb{E}}
\newcommand{\B}{\mathbb{ B}}
\newcommand{\M}{\mathbb{ M}}
\newcommand{\rz}{\mathbb{R}}
\newcommand{\sptext}[3]{\hspace{#1 em}\mbox{#2}\hspace{#3 em}}
\newcommand{\de}{{\delta}}

\newcommand{\Om}{{\Omega}}
\newcommand{\vare}{\varepsilon}
\newcommand{\s}{{\sigma}}
\newcommand{\al}{{\alpha}}
\newcommand{\be}{{\beta}}
\newcommand{\vph}{{\varphi}}

\newcommand{\half}{\frac{1}{2}}
\newcommand{\kla}{\left ( }
\newcommand{\mer}{\right ) }
\newcommand{\klae}{\left \{ }
\newcommand{\mere}{\right \} }
\newcommand{\noo}{\left \|}
\newcommand{\rrm}{\right \|}
\newcommand{\bet}{\left |}
\newcommand{\rag}{\right |}
\newcommand{\equa}{\begin{eqnarray*}}
\newcommand{\tion}{\end{eqnarray*}}
\newcommand{\equal}{\begin{eqnarray}}
\newcommand{\tionl}{\end{eqnarray}}
\newcommand{\bor}{\mathcal{B}}
\newcommand{\ftn}{\mathcal{F}}
\newcommand{\m}{\mathbbm{m}}
\newcommand{\mass}{\mathbb{P}}

\newcommand{\pl}{\ \ }                  
             
\newcommand{\les}{\hspace{-1.8em}}       
\newcommand{\Tau}{\mathcal{T}}

\begin{document}

\title[Fractional  smoothness, L\'evy processes and approximation]{A note on Malliavin fractional  smoothness for L\'evy processes and approximation}

\author{Christel Geiss}
\address{Department of Mathematics, University of Innsbruck,
         A-6020 Innsbruck, Technikerstra\ss e 13/7, Austria}
\email{christel.geiss@uibk.ac.at}

\author{Stefan Geiss}
\address{Department of Mathematics, University of Innsbruck,
         A-6020 Innsbruck, Technikerstra\ss e 13/7, Austria}
\email{stefan.geiss@uibk.ac.at}

\author{Eija Laukkarinen}
\address{Department of Mathematics and Statistics,
         University of Jyv\"askyl\"a,
         P.O. Box 35 (MAD),
         FIN-40014 Jyv\"askyl\"a,
         Finland}
\email{ei\-ja.laukkarinen@jyu.fi}
\thanks{The second and third author are supported by the Project 133914 of the Academy of Finland.}

\begin{abstract}
Assume a L\'evy process $(X_t)_{t\in [0,1]}$ that is an $L_2$-martingale and 
let $Y$ be either  its stochastic exponential  or $X$ itself. For certain integrands 
$\vph$ we investigate the behavior of 
\[
 \bigg  \|\int_{(0,1]} \vph_t dX_t - \sum_{k=1}^N
       v_{k-1} (Y_{t_k}-Y_{t_{k-1}}) \bigg \|_{L_2},
\]
where $v_{k-1}$ is $\ftn_{t_{k-1}}$-measurable, in dependence on the fractional 
smoothness in the Malliavin sense of $\int_{(0,1]} \vph_t dX_t$. A typical situation  where these techniques apply
occurs if  the stochastic integral is obtained by the Galtchouk-Kunita-Watanabe 
decomposition of some $f(X_1)$. Moreover, using the example $f(X_1)=\one_{(K,\infty)}(X_1)$ we show how  fractional smoothness 
depends on the distribution of  the L{\'e}vy process.
\end{abstract}
\maketitle


\section{Introduction}

We consider the quantitative Riemann approximation of stochastic
integrals driven by L\'evy processes and its relation to the fractional smoothness
in the Malliavin sense. Besides the interest on its own, the problem is of
interest for numerical algorithms and for Stochastic Finance.
To explain the latter aspect, assume a price process $(S_t)_{t\in [0,1]}$ 
given under der martingale measure by a diffusion
\[
  S_t =s_0 +  \int_0^t \sigma(S_r) dW_r,
\]
where $W$ is the Brownian motion and where usual conditions on $\sigma$
are imposed. For a polynomially bounded Borel function $f: \rz \to \rz$ we obtain a representation
\equal \label{cont-hedging}
 f(S_1) = V_0 + \int_0^1 \vph_t d S_t
\tionl
where  $(\vph_t)_{t\in [0,1)}$  is a continuous adapted process which can be obtained  via the gradient of a solution to a parabolic backward PDE  
related to $\sigma$ with terminal condition $f$. The process $(\vph_t)_{t\in [0,1)}$ is interpreted  as  a trading strategy.
In practice one can trade only finitely many times which corresponds to a
replacement of the stochastic integral in
\eqref{cont-hedging} by the sum
$ \sum_{k=1}^N   \vph_{t_{k-1}} ( S_{t_k} - S_{t_{k-1}})$
with $0=t_0 <  t_1 < \dots < t_N=1.$ The error 
\equal \label{sim-error}
\int_0^1 \vph_t d S_t - \sum_{k=1}^N \vph_{t_{k-1}} (S_{t_k}-S_{t_{k-1}}) 
\tionl 
caused by this replacement is often measured in $L_2$ and 
has been studied by various authors, for example by
Zhang \cite{zhang}, Gobet and Temam \cite{gobet-temam}, S. Geiss \cite{geiss}, S. Geiss and Hujo \cite{geiss-hujo}
and C. Geiss and S. Geiss \cite{geiss-geiss}.
For results concerning $L_p$ with $p\in (2,\infty)$
we refer to \cite{toivola}, 
the weak convergence is considered in \cite{geiss-toivola} and \cite{tankov-volt}
and by other authors.
In particular, if  $S$ is the Brownian motion
or the geometric Brownian motion,   S. Geiss and Hujo investigated in  \cite{geiss-hujo} 
the relation between the Malliavin fractional smoothness of $f(S_1)$ and the $L_2$-rate of the 
discretization error \eqref{sim-error}.
\smallskip

It is natural to extend these results to L\'evy processes.
A first step was done by
M. Brod\'en and  P. Tankov \cite{broden-tankov} (see Remark \ref{remark:broden-tankov}).
The aim of this paper is to extend results of \cite{geiss-hujo} 
into the following directions:
\medskip

(a) The Brownian motion and the geometric Brownian motion are generalized to
L\'evy processes $(X_t)_{t\in [0,1]}$ that are $L_2$-martingales and their
Dol\'eans-Dade exponentials  $S=\mathcal{E}(X)$, 
\[ S_t = 1 +\int_{(0,t]}S_{u^-} dX_u, \] 
respectively. For certain stochastic integrals
\[ F = \int_{(0,1]} \varphi_{s-} dX_s \]
and for $Y\in \{ X, \mathcal{E}(X)\}$ we study the connection of the
Malliavin fractional smoothness of $F$ (introduced by the real interpolation method)
and the behavior of
\equal  \label{opt-error} a_Y^{\rm opt}(F;(t_k)_{k=0}^N)
   = \inf \noo F - \sum_{k=1}^N v_{k-1} (Y_{t_k}-Y_{t_{k-1}}) \rrm_{L_2}, \tionl
where the infimum is taken over $\ftn_{t_{k-1}}$-measurable $v_{k-1}$ such that
$\ew v_{k-1}^2 (Y_{t_k}-Y_{t_{k-1}})^2 <\infty$ and where
$0=t_0<\cdots <t_N=1$ is a deterministic time-net.
\smallskip 

(b) In contrast to \cite{geiss-hujo}, where the reduction of the stochastic approximation
problem to a deterministic one is based on It\^o's formula and was done in \cite{geiss,geiss-geiss}, we prove an
analogous reduction in Theorems \ref{theorem:simple-approximation} and \ref{theorem:opt-sim}
by techniques based on the It\^o chaos decomposition.
\smallskip

(c) One more principal difference to  \cite{geiss-hujo} is the fact that
L\'evy processes do in general not satisfy the representation property
and therefore there are  $F\in L_2$ that cannot be approximated by sums of the form
$\sum_{k=1}^N v_{k-1} (Y_{t_k}-Y_{t_{k-1}})$ in $L_2$. As a consequence
we have to use the (orthogonal) Galtschouk-Kunita-Watanabe projection that 
projects $L_2$ onto the subspace $I(X)$ of stochastic integrals
$\int_{(0,1]} \lambda_s dX_s$ with $\ew\int_0^1 |\lambda_s|^2 ds <\infty$ that can be
defined in our setting as  the $L_2$-closure of
\begin{equation}\label{eqn:I(X)}
 \left \{ \sum_{k=1}^N v_{a_{k-1}} (X_{a_k}-X_{a_{k-1}}):
            v_{a_{k-1}}\in L_2(\ftn_{a_{k-1}}),
            {0=a_0<\cdots<a_N=1 
             \atop
             N=1,2,...} \right \}
\end{equation}
to deal with our approximation problem.
\bigskip

The paper is organized as follows. In Section \ref{sec:preliminaries} we recall some facts about 
real interpolation and L\'evy processes.
In Section \ref{sec:approximation} we investigate the discrete time approximation. 
The basic statement is Theorem \ref{theorem:simple-approximation} that reduces the 
stochastic approximation problem to a deterministic one in case of 
the Riemann-approximation \eqref{sim-error} (which we call simple approximation in the sequel).
The difference between the simple and optimal approximation \eqref{opt-error}
is shown in Theorem \ref{theorem:opt-sim} to be sufficiently small. Theorem \ref{theorem:lower_bound}
provides a lower bound for the optimal $L_2$-approximation. Finally, 
Theorems \ref{thm:equidistant-besov} and \ref{theorem:optimal-net-old} give the
connection to the Besov spaces defined by real interpolation.
We conclude  with Section \ref{sec:examples}  where we use the example $f(x)=\one_{(K,\infty)}(x)$
to demonstrate how the fractional smoothness depends on the underlying L\'evy process.


\section{Preliminaries}
\label{sec:preliminaries}

\subsection{Notation}  Throughout this paper we will use for $A,B,C \ge 0$ and $c \ge 1$ the notation 
$A\sim_c B $ for $\frac{1}{c} B \le A \le cB$ and $A=B \pm C$ for $B-C \le A \le B+C.$  The 
phrase {\em c\`adl\`ag} stands for a path which is right-continuous and has left limits. Given $q \in [1, \infty],$
the sequence space $\ell_q$ consists of all $\al =(\al_N)_{N \ge 1} \subseteq \rz$ such that $\| \al \|_{\ell_q} := \kla \sum_{N=1}^\infty |\al _N|^q \mer^{1/q}< \infty $
for $q<\infty$ and $\|\al \|_{\ell_\infty} := \sup_{N \ge 1} |\al_N| < \infty,$ respectively. 

\subsection{Real interpolation}
First we recall some facts about the real interpolation method.
\smallskip

\begin{definition} \rm 
For Banach spaces $X_1 \subseteq X_0$, where  $X_1$ is continuously
embedded into $X_0$, we define for $u>0$ the {\em K-functional}
\[
   K(u, x; X_0,X_1):= \inf_{x=x_0+x_1} \{\|x_0\|_{X_0} + u\|x_1\|_{X_1}\}.
\]
For $\theta \in (0,1)$ and $q \in [1,\infty]$ the real interpolation space
$(X_0,X_1)_{\theta, q}$
consists of all elements $x \in X_0 $ such that
$\|x\|_{(X_0,X_1)_{\theta, q}}  < \infty$ where
\[
        \|x\|_{(X_0,X_1)_{\theta, q}}
          :=\left \{ \begin{array}{ll}
\big [ \int_0^{\infty} [u^{-\theta }K(u, x; X_0,X_1)]^q \frac{du}{u}
           \big ]^\frac{1}{q}, & q \in [1, \infty) \\
&\\
 \sup_{u>0} u^{-\theta }K(u, x; X_0,X_1),  & q=\infty .
\end{array} \right .
\]
\end{definition}
\bigskip
The spaces $(X_0,X_1)_{\theta, q}$ equipped with $\|\cdot\|_{(X_0,X_1)_{\theta, q}}$
become Banach spaces and form a lexicographical scale, i.e. for
any $0< \theta_1<  \theta_2 <1$ and $ q_1, q_2 \in [1,\infty]$
it holds that
\[
          X_0  
\supseteq (X_0,X_1)_{ \theta_1, q_1} 
\supseteq (X_0,X_1)_{\theta_2, q_2} 
\supseteq (X_0,X_1)_{\theta_2, \min \{q_1,q_2\}}
\supseteq  X_1.
\]
For more information the reader is referred to \cite{bergh-lofstrom,bennet-sharp}.

\subsection{The spaces $\B_{2,q}^\theta(E)$}

\begin{definition} \rm \label{definition:Besov_sequence_spaces}
For a sequence of Banach spaces $E=(E_n)_{n=0}^\infty$ with
      $E_n\not = \{ 0 \}$ we let $\ell_2(E)$ and $d_{1,2}(E)$
      be the Banach spaces of all $a=(a_n)_{n=0}^\infty\in E$ such that
      \[    \|a\|_{\ell_2(E)}
         := \left ( \sum_{n=0}^\infty \| a_n \|_{E_n}^2 \right )^\frac{1}{2}
         \sptext{.4}{and}{.4}
            \| a\|_{d_{1,2}(E)}
         := \left ( \sum_{n=0}^\infty (n+1)\| a_n \|_{E_n}^2 \right )^\frac{1}{2}\!\!, \]
      respectively, are finite.
Moreover, for $\theta \in (0,1)$ and $q\in [1,\infty]$ we let
\[ \B_{2,q}^\theta(E) := \left \{ 
                         \begin{array}{rcl}
                         (\ell_2(E), d_{1,2}(E) )_{\theta,q} &:& \theta \in (0,1), q\in [1,\infty] \\ 
                          d_{1,2}(E)                       &:& \theta = 1, q=2
                         \end{array} \right . .  \]
\end{definition}
It can be shown that (cf. \cite[Remark A.1]{geiss-hujo})
\[ \| a \|_{\B_{2,2}^\theta(E)}^2 \sim_{c_\theta^2} \sum_{n=0}^\infty (n+1)^\theta \| a_n \|_{E_n}^2.
   \]
To describe the interpolation spaces $\B_{2,q}^\theta(E)$ we use two types of functions.
The first one is a generating function for $(\| a_n \|_{E_n}^2)_{n=0}^\infty$, i.e.
for $a=(a_n)_{n=0}^{\infty} \in \ell_2(E)$  we let
\[ 
T_a(t) := \sum_{n=0}^\infty \| a_n \|_{E_n}^2 t^n.
\]
The second function will be used to  describe our stochastic approximation in a deterministic way:
For $a\in \ell_2(E)$ and a deterministic time-net $\tau=(t_k)_{k=0}^N$ with  $0=t_0 \le \cdots \le t_N=1$ we let
\[
  A(a,\tau):= \bigg (\sum_{k=1}^N \int_{t_{k-1}}^{t_k} (t_k-t) (T_a)''(t)dt
  \bigg )^{\half}.
\]
For the formulation of the next two theorems  which will connect approximation 
properties with fractional smoothness special time nets are needed.
Given $\theta\in (0,1]$ and $N\ge 1$, we let $\tau_N^{\theta}$ be the time-net
\begin{equation}\label{eqn:nets}
   t_k^{N,\theta} := 1-\kla 1-\frac{k}{N}\mer^{\frac{1}{\theta}}
   \sptext{1}{for}{1} k=0,1, \ldots, N
\end{equation}
for which one has (see \cite[relation (4)]{geiss-toivola})
\begin{equation}\label{eqn:estimate_nets}
            \frac{|t_k^{N,\theta}- t|}{(1-t)^{1-\theta}} 
   \le      \frac{|t_k^{N,\theta}- t_{k-1}^{N,\theta}|}
            {(1-t_{k-1}^{N,\theta})^{1-\theta}}
   \le \frac{1}{\theta N}
   \sptext{1}{for}{1} k=1,...,N 
\end{equation}
and $t \in [t_{k-1}^{N,\theta}, t_k^{N,\theta}).$ For $\theta=1$ we obtain equidistant time-nets.
The following two theorems are taken from \cite{geiss-hujo}. For the convenience of the reader we comment about the proofs  
in Remark \ref{rem:proofs-geiss-hujo} below.

\begin{theorem}[\cite{geiss-hujo}]
\label{theorem:geiss-hujo-theta-q}
For $\theta\in (0,1)$, $q\in [1,\infty]$ and
$a=(a_n)_{n=0}^\infty\in \ell_2(E)$ one has
\[       \noo a \rrm_{\B_{2,q}^\theta(E)}
 \sim_c \| a \|_{\ell_2(E)} +
        \left \| \left ( N^{\frac{\theta}{2}-\frac{1}{q}} A(a,\tau^1_N) \right )_{N=1}^\infty
        \right \|_{\ell_q} \]
where $c\in [1,\infty)$ depends at most on $(\theta,q)$ and the expressions may be infinite.
\end{theorem}
\medskip

\begin{theorem} [{\cite{geiss-hujo}}]
\label{theorem:geiss-hujo-optimal-net} 
For $\theta \in (0,1]$  and $a=(a_n)_{n=0}^\infty\in \ell_2(E)$
the following assertions are equivalent:
\begin{enumerate}[{\rm (i)}]
\item \label{space} $ a \in  \B_{2,2}^{\theta}(E)$.
\item \label{T-primes} $ \int_0^1 (1-t)^{1-\theta} \, T''_F(t) dt < \infty$.
\item \label{a-sim-estimate} There exists a constant $c>0$ such that
      \[  A(a,\tau^\theta_N) \le \frac{c}{\sqrt{N}} 
          \sptext{1}{for}{1}
          N=1,2,\ldots  \]
\end{enumerate}
\end{theorem}

\begin{remark}\rm\label{rem:proofs-geiss-hujo}
We fix $a=(a_n)_{n=0}^\infty\in \ell_2(E)$ and $(\theta,q)$ according to 
Theorems \ref{theorem:geiss-hujo-theta-q} and \ref{theorem:geiss-hujo-optimal-net}. 
Then we let $\be_n := \| a_n \|_{E_n}$ and define 
$f=\sum_{n=0}^\infty \be_n h_n\in L_2(\rz,\gamma)$, where $\gamma$ is
the standard Gaussian measure and 
$(h_n)_{n=0}^\infty$ the orthonormal
basis of Hermite polynomials. As before, let
\[
  A(\beta,\tau):= \bigg (\sum_{k=1}^N \int_{t_{k-1}}^{t_k} (t_k-t) (T_\beta)''(t)dt
  \bigg )^{\half}
  \sptext{.7}{with}{.7}
  T_\beta(t) := \sum_{n=0}^\infty \beta_n^2 t^n.\]
Omitting the notation $(E)$ in the case $E=(\rz,\rz,...)$, we have
$\| a\|_{\ell_2(E)} = \|\beta\|_{\ell_2}$ and
$\| a\|_{d_{1,2}(E)} = \|\beta\|_{d_{1,2}}$.
Moreover, \cite[Theorem 2.2]{geiss-hujo} gives that
$\| a\|_{\B_{2,q}^\theta(E)} \sim_{c(\theta,q)} \|\beta\|_{\B_{2,q}^\theta}$
for $\theta\in (0,1)$ and $q\in [1,\infty]$ because of $T_a=T_\beta$.
Hence \cite[Lemmas 3.9 and 3.10, Theorem 3.5 (X=W)]{geiss-hujo} imply Theorem  \ref{theorem:geiss-hujo-theta-q} of this paper.
The equivalence of (i) and (iii) of Theorem \ref{theorem:geiss-hujo-optimal-net}
follows in the same way by  \cite[Lemmas 3.9 and  3.10, Theorem 3.2 (X=W)]{geiss-hujo}.
Finally, the equivalence of (i) and (ii) of Theorem \ref{theorem:geiss-hujo-optimal-net} 
is a consequence of the proof of  \cite[Theorem 3.2 (X=W)]{geiss-hujo}.
\end{remark}

\subsection{L{\'e}vy processes}
We follow the setting and presentation of \cite[Section 1.1]{sole-utzet-vives2}
and assume a {\em square integrable mean zero} L{\'e}vy process 
$X=(X_t)_{t\in [0,1]}$ on a stochastic basis 
$(\Omega,\ftn,\mass,(\ftn_t)_{t\in [0,1]})$ satisfying the usual assumptions, i.e.
$(\Omega,\ftn,\mass)$ is complete where the filtration   $(\ftn_t)_{t\in [0,1]}$ is the 
augmented natural filtration of $X$ and therefore  right-continuous 
and $\ftn:= \ftn_1$ is assumed without loss of generality.
The L\'evy measure $\nu$ with  $\nu(\{0\})=0$ satisfies 
\[ \int_\rz x^2 \nu(dx) < \infty \] 
by the square integrability of $X$ (see \cite[Theorem 25.3]{sato}). 
Let $N$ be the associated Poisson random measure and
$d \tilde{N}(t,x) = d N(t,x) - d t d \nu(x)$
be the compensated Poisson random measure.
The L\'evy-It\^o decomposition (see \cite[Theorem 19.2]{sato}) can be 
written under our assumptions as 
\[ X_t = \sigma W_t 
      + \int_{(0,t] \times \rz\setminus \{0\}} x
      \tilde{N}(ds,dx). \]
We introduce the finite measures $\mu$ on $\bor(\rz)$ and
$\m$ on $\bor([0,1]\times\rz)$ by
\equa
 \mu(dx)    &:= & \s^2  \de_0(dx) + x^2  \nu(dx), \\
  \m(dt,dx) &:= & d t \mu(dx),
\tion
where we agree about $\mu(\rz) > 0$ to avoid pathologies.
For $B\in \bor((0,1]\times\rz)$ we define the random measure
\[ M(B) := \s \int_{\{t\in (0,1]:(t,0)\in B\}} d W_t
           + \int_{B\cap ((0,1] \times (\rz\setminus \{0\}))} x  \tilde{N}(dt,dx) \]
and let
\[ L^n_2 := L_2 ( ([0,1]\times\rz)^n, \bor(([0,1]\times \rz)^{n}), \m^{\otimes n})
   \sptext{1}{for}{1}
   n\ge 1. \]
By \cite[Theorem 2]{ito} there is the  chaos decomposition
$$L_2 := L_2(\Om,\ftn,\PP) = \bigoplus_{n=0}^\infty I_n(L^n_2), $$
where $I_0(L^0_2)$  is the space of the a.s.~constant random variables and
$I_n(L^n_2):=\{I_n(f_n):f_n\in L^n_2\}$ for $n=1,2,\ldots$ 
and $I_n(f_n)$ denotes the multiple integral w.r.t.
the random measure $M.$ For properties of the multiple integral see \cite[Theorem 1]{ito}. 
Especially, $ \|I_n(f_n)\|_{L_2}^2 =n! \|\tilde{f}_n\|_{L_2^n}^2$ and
\equa
\|F\|^2_{L_2} = \sum_{n=0}^\infty n! \|\tilde{f}_n\|^2_{L^n_2}
\tion
with $\tilde{f}_n$ being the symmetrization of $f_n$, i.e.
\equa
\tilde{f}_n (z_1, \ldots, z_n)= \frac{1}{n!} \sum f_n(z_{\pi(1)}, \ldots,
   z_{\pi(n)})
\tion
for all $ z_i=(t_i,x_i) \in [0,1]\times\rz,$  where the sum is taken over all permutations $\pi$
of $\{1, \ldots, n\}.$  For 
$F \in L_2$ the $L_2$-representation
$$F = \sum_{n=0}^\infty I_n(\tilde{f}_n),$$
with $I_0(f_0) = \ew F$ a.s. is unique (note that
$I_n(f_n) = I_n(\tilde{f}_n)$ a.s.). 

\subsection{Besov spaces}
Here we recall the construction of Besov spaces (or spaces of random variables of
fractional smoothness) based on the above chaos expansion.

\begin{definition}\rm 
Let  $\DD_{1,2}$ be the space of all 
$F = \sum_{n=0}^\infty I_n(f_n) \in L_2$ such that
\[    \|F\|^2_{\DD_{1,2}} 
   := \sum_{n=0}^\infty (n+1) \|I_n(f_n) \|^2_{L_2} < \infty. \]
Moreover,
\[ \B_{2,q}^\theta := \left \{ 
                         \begin{array}{rcl}
                         (L_2, \DD_{1,2})_{\theta,q} &:& \theta \in (0,1), q\in [1,\infty] \\ 
                          \DD_{1,2}                 &:& \theta = 1, q=2
                         \end{array} \right . .  \]
\end{definition}
\medskip

\subsection{The space of the random variables to approximate}
We will approximate random variables from the following space $\M$:

\begin{definition} \rm 
\label{definition:M}
The closed subspace $\M \subseteq L_2$ consists of all mean zero $F\in L_2$ such that there exists 
a representation  
\[ F= \sum_{n=1}^{\infty}  I_n (f_n) \]
with symmetric $f_n$ such that there are $h_0\in \rz$ and symmetric $h_n\in L_2(\mu^{\otimes n})$
for $n\ge 1$ with 
\[  f_n((t_1,x_1),...,(t_n,x_n)) = h_{n-1}(x_1,...,x_{n-1})
    \sptext{1}{for}{.5}
    0<t_1<\cdots<t_n<1. \]
The orthogonal projection onto $\M$ is denoted by $\Pi:L_2\to \M\subseteq L_2$.
\end{definition}
\medskip
Let us summarize some facts about the space $\M$:
\medskip

(a) {\bf Representation of $\Pi$.} For
\[ G= \sum_{n=0}^\infty I_n (\alpha_n)  \in  L_2 \]
with symmetric $\alpha_n\in L_2^n$  one  computes  the functions $h_n$  of the projection  $F= \Pi(G)$  by
\equal \label{hn-1}
&   & \les h_{n-1}(x_1,...,x_{n-1}) \nonumber \\
& = & n! \int_0^1 \int_0^{t_{n-1}}\!\!\!... \int_0^{t_2}\int_\rz \alpha_n((t_1,x_1),...,(t_{n-1},x_{n-1}),
      (t_n,x_n)) \nonumber\\
&   & \times \frac{\mu(dx_n)}{\mu(\rz)}   dt_1 \cdots dt_n \quad \text{ for } n\ge 1.
\tionl
\medskip

(b) {\bf Integral representation of the elements of $\M$.}
Given $F\in \M$ with a representation like in Definition \ref{definition:M} 
(the functions $h_n$  are unique as elements of $L_2(\mu^{\otimes n})$),
we define the martingale 
$\varphi=(\varphi_t)_{t\in [0,1)}$ by the $L_2$-sum
\begin{equation}\label{eqn:vphi}
        \varphi_t 
    :=  h_0 + \sum_{n=1}^\infty (n+1) I_n \left (h_n \one^{\otimes n}_{(0,t]}\right ),
\end{equation}
which we will assume to be path-wise {\em c\`adl\`ag}.
It follows that
\equa
      \|  \varphi_t \|_{L_2}^2 
& = & h_0^2 + \sum_{n=1}^\infty (n+1)^2 n! t^n \| h_n \|_{L_2(\mu^{\otimes n})}^2 \\
& = & h_0^2 + \frac{1}{\mu(\rz)} \sum_{n=1}^\infty (n+1)^2 n! t^n \| f_{n+1} \|_{L_2^{n+1}}^2 \\
& = & h_0^2 + \frac{1}{\mu(\rz)} \sum_{n=1}^\infty t^n (n+1) \| I_{n+1} (f_{n+1}) \|_{L_2}^2 \\
\tion
so that
\begin{equation}\label{eqn:D12_boundedness_phi-martingale}
   \mu(\rz) \sup_{t\in [0,1)} \|\vph_t \|_{L_2}^2 + \| F\|_{L_2}^2 
= \sum_{n=0}^\infty (n+1) \| I_n(f_n) \|_{L_2}^2. 
\end{equation}
Moreover, for $t\in [0,1]$ we get that, a.s.,
\[
      F_t
 :=   \ew (F|\ftn_t) 
  =   \int_{(0,t]} \varphi_{s-} dX_s. \]
This is analog to the Brownian motion case considered in  \cite{geiss-geiss} and \cite{geiss-hujo},
where the representation $F= \ew F + \int_{(0,1]} \vph_s dB_s$ was used together with the regularity 
assumption that $(\vph_s)_{s\in [0,1)}$ is a martingale  or close to a martingale in some sense.
\medskip

(c) {\bf Basic examples for elements for $\M$} are taken from Lemma \ref{lemma:kgw-projection} below:
Let $\Pi_X:L_2\to I(X) \subseteq L_2$ be the orthogonal projection onto $I(X)$ defined in
(\ref{eqn:I(X)}) and let $f:\rz\to\rz$ be a Borel function with $f(X_1)\in L_2$, then
\[ \Pi_X (f(X_1)) = \Pi (f(X_1)). \]
This means the elements of $\M$ occur naturally when applying the Galtchouk-Kunita-Watanabe projection.
It should be noted, that in the case that $\sigma=0$ and $\nu=\al \delta_{x_0}$ with $\al>0$ and $x_0\in \rz\setminus \{0\}$ we have
a chaos decomposition of the form $f(X_1) = \ew f(X_1) + \sum_{n=1}^\infty \beta_n I_n(\one_{(0,1]}^{\otimes n})$ with $\beta_n\in \rz$, so that already $f(X_1)\in \M$.
  
\subsection{Dol\'eans-Dade stochastic exponential}

\begin{definition}\rm
For $0\le a \le t \le 1$ we let
\[ S_t^a := 1 + \sum_{n=1}^\infty \frac{I_n(\one_{(a,t]}^{\otimes n})}{n!}, \]
where we can assume that all paths of $(S_t^a)_{t\in [a,1]}$ are 
c\`adl\`ag for any fixed $a\in [0,1]$. In particular, we let
$S=(S_t)_{t\in [0,1]} := (S_t^0)_{t\in [0,1]}$.
\end{definition}
The following lemma is standard and we omit its proof.

\begin{lemma}\label{lemma:properties_local_S}
For $0\le a \le t \le 1$ one has that
\begin{enumerate}[{\rm (i)}]
\item $S_t^a = 1 + \int_{(a,t]} S_{u-}^a dX_u$ a.s.,
\item $S_t = S_t^a S_a$ a.s.,
\item $S_t^a$ is independent from $\ftn_a$ and $\ew (S_t^a)^2 = e^{\mu(\rz)(t-a)}$.
\end{enumerate}
\end{lemma}


\section{Approximation of stochastic integrals}
\label{sec:approximation}

In the sequel we will use
\[ \Tau_N := \{ \tau=(t_k)_{k=0}^N: 0=t_0<\cdots<t_N=1 \}
   \sptext{1}{and}{1}
   \Tau := \bigcup_{N=1}^\infty \Tau_N \]
as sets of deterministic time-nets and define $|\tau|:=\max_{1 \le k \le N} |t_k-t_{k-1}|.$
We will consider the following approximations
of a random variable $F\in \M$ with respect to the processes $X$ and $S$:
\pagebreak

\begin{definition}\rm
For  $N\ge 1$, $Y\in \{ X,S \}$, $F=\int_{(0,1]}\vph_{s-} dX_s \in \M,$  $A=(A_k)_{k=1}^N\subseteq \ftn$ and $\tau \in \Tau_N $ we let
\begin{enumerate}[(i)]
\item $a^{\rm sim}_S(F;\tau,A) :=  \left \| F  - \sum_{k=1}^N \vph_{t_{k-1}} \one_{A_k} (S_{t_k}^{t_{k-1}} - 1) 
      \right \|_{L_2}$,
\item $a^{\rm opt}_Y(F;\tau) :=  \inf \left \| F -  \sum_{k=1}^N v_{k-1} (Y_{t_k}-Y_{t_{k-1}}) \right \|_{L_2}$,
      where the infimum is taken over all $\ftn_{t_{k-1}}$-measurable $v_{k-1}:\Omega\to\rz$ such that 
      $\ew|v_{k-1} (Y_{t_k}-Y_{t_{k-1}})|^2<\infty$.
\end{enumerate}
\end{definition}
\medskip

\begin{remark}\rm
\begin{enumerate}[{\rm (i)}]
\item The definition of $a^{\rm sim}_S$ takes into account the additional sets $(A_k)_{k=1}^N$ 
      to avoid problems with the case that $S$ vanishes. These extra sets $A$ in 
      $a^{\rm sim}_S(F;\tau,A)$ play different roles 
      in Theorem \ref{theorem:simple-approximation},
       Theorem \ref{theorem:opt-sim}, and 
      in  Theorems \ref{theorem:lower_bound}, 
                 \ref{thm:equidistant-besov} and
                 \ref{theorem:optimal-net-old}.
      To recover a more standard form of $a^{\rm sim}_S$ assume that $(S_t)_{t \in [0,1]}$ 
      and $(S_{t-})_{t \in [0,1]}$ are positive so that we can write 
       \[ F = \int_{(0,1]} \psi_{u-} (S_{u-} dX_u)
                 \sptext{1}{with}{1}
                  \psi_u := \frac{\vph_u}{S_u}
       \]
      and obtain that 
      \equa   F -  \sum_{k=1}^N \vph_{t_{k-1}}(S_{t_k}^{t_{k-1}} - 1) 
             &=&  F -  \sum_{k=1}^N \psi_{t_{k-1}} S_{t_{k-1}}  (S_{t_k}^{t_{k-1}} - 1) \\  
             &=&  F -  \sum_{k=1}^N \psi_{t_{k-1}}  (S_{t_k} - S_{t_{k-1}})
      \tion   
      which is what one expects.
\item In the sequel the crucial assumption will be
      \[ \Omega = \{ S_t \not = 0 \}
         \sptext{1}{for all}{1}
        t\in [0,1]. \]
      This can be achieved by the condition $\nu((-\infty,-1])=0$ which implies 
      the almost sure positivity of $S$ and we can adjust $S$ on a set of 
      measure zero; see \cite[Theorem I.4.61]{jacod-shirjaev} and \cite[Theorem 19.2]{sato}.
\end{enumerate}
\end{remark}
\medskip

Because of the martingale property of $(\vph_t)_{t\in [0,1)}$ 
it is easy to check that
\[    a^{\rm opt}_X(F;\tau) 
   =  \left \| F - \sum_{k=1}^N \vph_{t_{k-1}} (X_{t_k}-X_{t_{k-1}}) 
      \right \|_{L_2} \]
so that for $Y=X$ the simple and optimal approximation coincide.
The theorem below gives a description of the simple approximation
by a function $H_Y(t)$ that describes, in some sense, the curvature of $F\in \M$
with respect to $Y$. 

\pagebreak
\begin{theorem}
\label{theorem:simple-approximation}
Let $F\in \M$,
\[ H_Y^2(t) := \mu(\rz) \sum_{n=1}^\infty  n  n!  
               t^{n-1} \| A_n^Y \|_{L_2(\mu^{\otimes n})}^2 \]
with
\begin{multline*}
 A_n^Y(x_1,...,x_n) \\
         := \left \{ \begin{array}{lcl}
                   (n+1) h_n(x_1,...,x_n)     &:& Y=X \\
                   (n+1) h_n(x_1,...,x_n)-h_{n-1}(x_1,...,x_{n-1}) &:& Y=S
                   \end{array} \right ..
\end{multline*}
Then, for $\tau\in \Tau$, one has 
\equa 
      a_X^{\rm opt}(F;\tau)   
& = & \left (\sum_{k=1}^N \int_{t_{k-1}}^{t_k} (t_k-t) H_X^2(t)dt \right )^\frac{1}{2},\\
      a^{\rm sim}_S(F;\tau,\Omega^N)    
&\sim_c& \left (\sum_{k=1}^N \int_{t_{k-1}}^{t_k} (t_k-t) H_S^2(t)dt \right )^\frac{1}{2},
\tion
where in the last equivalence
$|\tau|< 1/\mu(\rz)$ and $c := (1-\sqrt{ \mu(\rz) |\tau|})^{-1}$ and $\Om^N = (\Om, \ldots, \Om).$
\end{theorem}

\bigskip

\begin{proof}
\underline{Case $Y=X$}: We get that
\equa
      \ew |\varphi_t-\varphi_{t_{k-1}}|^2
& = & \sum_{n=1}^\infty (t^n-t_{k-1}^n) (n+1)^2 n! \| h_n\|_{L_2(\mu^{\otimes n})}^2 \\
& = & \sum_{n=1}^\infty (n+1)^2 n\,n! \int_{t_{k-1}}^t u^{n-1} du  \| h_n\|_{L_2(\mu^{\otimes n})}^2 \\
& = & \frac{1}{\mu(\rz)} \int_{t_{k-1}}^t H_X^2(u) du 
\tion
which implies for $a^{\rm sim}_X(F;\tau)=a^{\rm opt}_X(F;\tau)=:a_X(F;\tau)$ that 
\equa
      |a_X(F;\tau)|^2  
   &=&  \mu(\rz) \sum_{k=1}^N \int_{t_{k-1}}^{t_k}\ew |\varphi_t-\varphi_{t_{k-1}}|^2 dt\\
   &=&  \sum_{k=1}^N \int_{t_{k-1}}^{t_k} (t_k-u) H_X^2(u) du. 
\tion
\underline{Case $Y=S$}: Here we get that
\equa
&   & \les a^{\rm sim}_S(F;\tau,\Omega^N)  \\
& = & \left ( \mu(\rz)
      \sum_{k=1}^N \int_{t_{k-1}}^{t_k}
      \ew\left|\varphi_t-\varphi_{t_{k-1}} S_{t-}^{t_{k-1}} \right|^2 dt \right ) ^\frac{1}{2} \\
& = & \bigg ( \mu(\rz)
      \sum_{k=1}^N \int_{t_{k-1}}^{t_k}
      \ew\bigg|  \left [\varphi_t- \varphi_{t_{k-1}} - \int_{(t_{k-1},t]} \varphi_{u-} dX_u 
                 \right ] \\
&   &          + \left [\int_{(t_{k-1},t]} \varphi_{u-} dX_u - \varphi_{t_{k-1}}(S_{t-}^{t_{k-1}}-1)
                 \right ]
      \bigg |^2 dt \bigg ) ^\frac{1}{2} \\
& = & \bigg ( \mu(\rz)
      \sum_{k=1}^N \int_{t_{k-1}}^{t_k}
      \ew \left [\varphi_t- \varphi_{t_{k-1}} - \int_{(t_{k-1},t]} \varphi_{u-} dX_u 
                 \right ]^2 dt \bigg ) ^\frac{1}{2} \\
&   &         \pm  \bigg ( \mu(\rz)
      \sum_{k=1}^N \int_{t_{k-1}}^{t_k}
      \ew \left [\int_{(t_{k-1},t]} \varphi_{u-} dX_u - \varphi_{t_{k-1}}(S_{t-}^{t_{k-1}}-1)
                 \right ]^2 dt \bigg ) ^\frac{1}{2}
\tion
where
\equa
&   &\les  \bigg ( \mu(\rz)
      \sum_{k=1}^N \int_{t_{k-1}}^{t_k}
      \ew \left [\int_{(t_{k-1},t]} \varphi_{u-} dX_u - \varphi_{t_{k-1}}(S_{t-}^{t_{k-1}}-1)
                 \right ]^2 dt \bigg ) ^\frac{1}{2} \\
&\le& \sqrt{|\tau|}
           \bigg ( \mu(\rz)
      \sum_{k=1}^N 
      \ew \left [\int_{(t_{k-1},t_k]} \varphi_{u-} dX_u - \varphi_{t_{k-1}}(S_{t_k}^{t_{k-1}}-1)
                 \right ]^2  \bigg ) ^\frac{1}{2} \\     
& = & \sqrt{|\tau|\mu(\rz)}  a^{\rm sim}_S(F;\tau,\Omega^N)
\tion
where we used $S_{t-}^{t_{k-1}}=S_t^{t_{k-1}}$ a.s. for $t \in (t_{k-1},t_k]$ and the martingale property of $\int_{(t_{k-1},t]} \varphi_{u-} dX_u - \varphi_{t_{k-1}}(S_t^{t_{k-1}}-1)$.
Finally,
\equa
&   &\les  \bigg ( \mu(\rz)
      \sum_{k=1}^N \int_{t_{k-1}}^{t_k}
      \ew \left [\varphi_t- \varphi_{t_{k-1}} - \int_{(t_{k-1},t]} \varphi_{u-} dX_u 
                 \right ]^2 dt \bigg ) ^\frac{1}{2} \\
& =  & \bigg ( \mu(\rz)
      \sum_{k=1}^N \int_{t_{k-1}}^{t_k}
      \ew \left [(\varphi_t- \varphi_{t_{k-1}}) - 
                 (      F_t -      F_{t_{k-1}})  
                 \right ]^2 dt \bigg ) ^\frac{1}{2} \\
& =  & \bigg (\sum_{k=1}^N \int_{t_{k-1}}^{t_k} \int_{t_{k-1}}^t H_S^2(u) du
      dt \bigg ) ^\frac{1}{2}.
\tion
\end{proof}
The next theorem states that the simple and optimal approximation are equivalent
whenever $A_k := \{ S_{t_{k-1}} \not = 0 \}$ is taken. 
\pagebreak

\begin{theorem}\label{theorem:opt-sim}
For $F\in \M$ and $\tau\in \Tau$ one has that
\[     | a_S^{\rm sim}(F;\tau,A)- a^{\rm opt}_S(F;\tau)| 
   \le c \big [ |\tau| \|F\|_{L_2} + \sqrt{|\tau|} a_X^{\rm opt}(F;\tau) \big ] \]
where $c>0$ depends on $\mu$ only and $A_k := \{ S_{t_{k-1}} \not = 0 \}$.
\end{theorem}

\begin{proof}
(a) In the first step we determine an optimal sequence of $(v_k)_{k=1}^{N-1}$. 
For $0\le a < b \le 1$ we get from Lemma \ref{lemma:properties_local_S} that
\equa
&   & \inf \klae  \left \| v (S_b -S_a) - \int_{(a,b]} \vph_{u-} dX_u \right \|_{L_2}  :
      { v \mbox{ is $\ftn_a$-measurable } \atop  \ew |v (S_b -S_a)|^2 < \infty} \mere \\
& = & \inf \klae  \left \| v S_a (S_b^a -1)- \int_{(a,b]} \vph_{u-} dX_u \right \|_{L_2}  :
      { v \mbox{ is $\ftn_a$-measurable } \atop  \ew |v S_a|^2 < \infty} \mere \\
& = & \inf \klae  \left \|  \overline{v} \one_{\{S_a \not = 0 \}} (S_b^a -1) - 
      \int_{(a,b]} \vph_{u-} dX_u  \right \|_{L_2}  \!\! : \!\!
      { \overline{v} \mbox{ is $\ftn_a$-measurable } \atop  \ew |\overline{v}|^2 < \infty}\!\!\! \mere.
\tion
The infimum is obtained with
\begin{multline*}
  \overline{v} =  \\ \frac{\ew \kla \int_a^b \vph_{t-} S_{t-}^a     dt | \ftn_a \mer}
                      {\ew \kla \int_a^b           (S_{t-}^a)^2 dt | \ftn_a \mer} 
                  =  \frac{\ew \kla \int_a^b \vph_{t} S_{t}^a     dt | \ftn_a \mer}
                      {\int_a^b \ew  (S_{t}^a)^2 dt} 
                 =:  \frac{ \ew \kla \int_a^b \vph_{t} S_{t}^a     dt | \ftn_a \mer }{\kappa(a,b)}
\end{multline*}
and
\[ v := \left \{ \begin{array}{lcl}
                 \frac{1}{S_a \kappa(a,b)} \ew \kla \int_a^b \vph_{t} S_{t}^a     dt | \ftn_a \mer &:& S_a \not = 0 \\
                                                                                                0 &:& S_a      = 0
                 \end{array} \right .
\]
where we used that
\equal \label{almost-sure}
 \vph_{t-} = \vph_t \quad \text{a.s. and } S_{t-}^a =S_{t}^a \quad \text{a.s. on  } (a,b].
\tionl
\bigskip

(b) Now it holds that
\equa
&   & \les | a_S^{\rm sim}(F;\tau,A)- a^{\rm opt}_S(F;\tau)| \\
& = &  \bigg | \left \| F - \ew F - \sum_{k=1}^N \vph_{t_{k-1}} \one_{A_k}  (S_{t_k}^{t_{k-1}} - 1) \right \|_{L_2} \\
&   & - \left \| F - \ew F - \sum_{k=1}^N v_{k-1} (S_{t_k} - S_{t_{k-1}}) \right \|_{L_2} \bigg |\\
&\le&   \left \|   \sum_{k=1}^N [\vph_{t_{k-1}} - v_{k-1} S_{t_{k-1}}] (S_{t_k}^{t_{k-1}} - 1) \one_{A_k}
        \right \|_{L_2} \\
& = & \kla \sum_{k=1}^N \| [\vph_{t_{k-1}} - v_{k-1} S_{t_{k-1}}] \one_{A_k}\|_{L_2}^2 
      [e^{\mu(\rz)(t_k-t_{k-1})} - 1] \mer^\frac{1}{2}.
\tion 
Moreover (using again \eqref{almost-sure}) we have
\equa
&   &  \les  \| [\vph_{t_{k-1}} - v_{k-1} S_{t_{k-1}}] \one_{A_k}\|_{L_2} \\
&\le&    \| \vph_{t_{k-1}} \kla 1 - \frac{t_k-t_{k-1}}{\kappa(t_{k-1},t_k)} \mer \one_{A_k} \|_{L_2} \\
&   &  + \noo \frac{ \one_{A_k}}{\kappa(t_{k-1},t_k)}
              \ew \kla \int_{t_{k-1}}^{t_k} (\vph_t - \vph_{t_{k-1}})(S_t^{t_{k-1}}-1) dt | \ftn_{t_{k-1}} \mer \rrm_{L_2}.
\tion
The first term on the right-hand side can be bounded from above by $\mu(\rz) (t_k-t_{k-1}) \| \vph_{t_{k-1}} \one_{A_k} \|_{L_2}$.
For the second term  we let $a=t_{k-1}<t_k=b$ and $\lambda_t = \one_{A_k} (\vph_t - \vph_{t_{k-1}})$ and obtain
\equa
&   & \les \ew \kla \int_a^b \lambda_t (S_t^a-1) dt \Big | \ftn_a \mer \\
&\le& \kla \ew \kla \int_a^b |\lambda_t|^2 dt \Big| \ftn_a \mer \mer^\frac{1}{2} 
      \kla \ew \kla \int_a^b (S_t^a-1)^2   dt\Big | \ftn_a \mer \mer^\frac{1}{2} \\
& = & \kla \ew \kla \int_a^b |\lambda_t|^2 dt \Big| \ftn_a \mer \mer^\frac{1}{2} 
      \kla          \int_a^b \| S_t^a-1 \|_2^2   dt \mer^\frac{1}{2} \\
&\le& \kla \ew \kla \int_a^b |\lambda_t|^2 dt \Big| \ftn_a \mer \mer^\frac{1}{2} 
      \sqrt{\frac{\mu(\rz)}{2}} \kappa(a,b)
\tion
where the last inequality follows from
\equa
      \int_a^b \| S_t^a-1 \|_2^2   dt
& = & \int_a^b \mu(\rz) \kappa(a,t) dt \\
&\le& \int_a^b \mu(\rz) \kappa(a,t) \kla \frac{d}{dt} \kappa(a,t) \mer dt \\
& = & \frac{\mu(\rz)}{2} \kappa(a,b)^2.
\tion
Hence
\equa
&   & \les \| [\vph_{t_{k-1}} - v_{k-1} S_{t_{k-1}}] \one_{A_k}\|_{L_2} \\
&\le& \mu(\rz) (t_k-t_{k-1}) \| \vph_{t_{k-1}} \one_{A_k} \|_{L_2} \\
&   & + \sqrt{\frac{\mu(\rz)}{2}} \kla 
        \int_{t_{k-1}}^{t_k} \|  \one_{A_k}(\vph_t - \vph_{t_{k-1}}) \|_{L_2}^2  dt 
                                  \mer^\frac{1}{2}.
\tion
Using $e^{\mu(\rz)(t_k-t_{k-1})}-1 \le \mu(\rz) e^{\mu(\rz)}(t_k-t_{k-1})$ we conclude with
\equa
&   & \les | a_S^{\rm sim}(F;\tau,A)- a^{\rm opt}_S(F;\tau)| \\
&\le& \kla \sum_{k=1}^N \left [ 
                             \mu(\rz) (t_k-t_{k-1}) \| \vph_{t_{k-1}} \one_{A_k} \|_{L_2}   
                                             \right ]^2 
       \mu(\rz) e^{\mu(\rz)}(t_k-t_{k-1}) \mer^\frac{1}{2} \\
&    & \hspace*{-2em}
        + \kla \sum_{k=1}^N \left [  
                           \frac{\mu(\rz)}{2} \int_{t_{k-1}}^{t_k} 
                              \| \one_{A_k}(\vph_t - \vph_{t_{k-1}}) \|_{L_2}^2 dt 
                   \right ] \mu(\rz) e^{\mu(\rz)}(t_k-t_{k-1}) \mer^\frac{1}{2} \\
&\le& |\tau| \mu(\rz) e^{\mu(\rz)/2} \|F\|_{L_2}
       + \sqrt{|\tau|}  \sqrt{\frac{\mu(\rz)}{2}} e^{\mu(\rz)/2} a_X^{\rm opt}(F;\tau).
\tion
\end{proof}

Now we show that $1/\sqrt{N}$ is the lower bound for our approximation if time-nets
of cardinality $N+1$ are used.
\bigskip

\begin{theorem}\label{theorem:lower_bound}
Let $F \in \M$ and $Y\in \{ X,S \}$, where in the case $X=S$ we assume that 
$\Omega=\{S_t\not = 0 \}$ for all $t\in [0,1]$.
Unless there are $a,b\in \rz$ such that  $F=a+bY_1$ a.s., one has that
\[ \liminf_{N\to\infty} \sqrt{N} \left [ \inf_{\tau_N \in \Tau_N} a_Y^{\rm opt}(F;\tau_N)\right ] >0. \]
\end{theorem}
\smallskip
\begin{proof}
\underline{Case $Y=X$}: We have $H_X(t)=0$ for some $t\in (0,1)$ if and only if
$h_n=0$ $\mu^{\otimes n}$ a.e. for all $n=1,2,...$ which implies that
$F=I_1(f_1)= I_1(h_0) = h_0 X_1$. This means that our assumption on $F$ 
implies that $H_X(t)>0$ for all $t\in (0,1)$. Consequently, 
Theorem \ref{theorem:simple-approximation} gives for any fixed $s\in (0,1)$ that
\equa
      N \big | a_X^{\rm opt}(F;\tau_N) \big |^2
& = & N\sum_{k=1}^N\int_{t_{k-1}}^{t_k} (t_k-t)H_X^2(t) dt\\
&\ge& N \int_s^1 \bigg [ \sum_{k=1}^N(t_k-t)\one_{[t_{k-1},t_k)}(t) H_X^2(s) \bigg ]dt\\
& = & \frac{1}{2} H_X^2(s) N\sum_{k=1}^N (t_k \vee s - t_{k-1}\vee s)^2\\
&\ge& \frac{1}{2} H_X^2(s) (1-s)^2
\tion
which proves the statement for $Y=X$.
\smallskip

\underline{Case $Y=S$}:
Similarly as in the previous case
our assumption on $F$ implies that $H_S(t)>0$ for all $t\in (0,1)$.
In fact, assuming that $H_S(t)=0$ for some $t\in (0,1)$ implies 
\[  (n+1) h_n(x_1,...,x_n) = h_{n-1}(x_1,...,x_{n-1})  \hspace{1em} \mu^{\otimes n} \mbox{-a.e.} \]
for all $n=1,2,...$. By induction we derive that 
\[ h_n = \frac{h_0}{(n+1)!}  \hspace{1em} \mu^{\otimes n}\mbox{-a.e.} 
   \sptext{1}{for}{1}
   n\ge 0 \]
so that
$f_n = h_0/n!$ $m^{\otimes n}$-a.e. for $n\ge 1$. This would give that
$F= h_0(S_1-1)$ a.s.
\medskip

Hence applying Theorem \ref{theorem:simple-approximation} as in the case
$Y=X$ implies that there is an
$\vare>0$ such that 
\[  \sqrt{N} a_S^{\rm sim}(F;\tau_N,\Omega^N) \ge \vare >0  
    \quad \text{for all } \tau_N\in \Tau_N  \text{ with } |\tau_N| \le \frac{1}{2\mu(\rz)}. \]
For an arbitrary $N \ge 1$ and $\tau_N\in \Tau_N$ Theorem \ref{theorem:opt-sim} gives
\[       a^{\rm opt}_S(F;\tau_N)
   \ge  a_S^{\rm sim}(F;\tau_N,\Omega^N) - c_{(\ref{theorem:opt-sim})} 
        \big [ |\tau_N| \|F\|_{L_2} + \sqrt{|\tau_N|} a_X^{\rm opt}(F;\tau_N) \big ]. \]
Letting $\widetilde{\tau}_N := \tau_N \cup \{k/N:k=1,...,N-1\} 
\in \bigcup_{k=N}^{2N-1}\Tau_k$, $N \ge 2 \mu(\rz) \vee 2$ implies
$|\widetilde{\tau}_N|\le 1/N\le 1/(2\mu(\rz))$ and
\equa 
&   & \les \sqrt{N} a^{\rm opt}_S(F;\tau_N) \\
&\ge&  \sqrt{N} a^{\rm opt}_S(F;\widetilde{\tau}_N)\\
&\ge&  \sqrt{N} \frac{\vare}{\sqrt{2N}} 
        - c_{(\ref{theorem:opt-sim})} \sqrt{N}
        \Big [   |\widetilde{\tau}_N| \|F\|_{L_2} 
               + \sqrt{|\widetilde{\tau}_N|} a_X^{\rm opt}(F;\widetilde{\tau}_N) \Big ]\\
&\ge&  \frac{\vare}{\sqrt{2}} 
        - c_{(\ref{theorem:opt-sim})}
        \left [   \frac{\|F\|_{L_2}}{\sqrt{N}}
               + a_X^{\rm opt}(F;(k/N)_{k=0}^N) \right ].              
\tion
The convergence $a_X^{\rm opt}(F;(k/N)_{k=0}^N) \to 0$ as $N \to \infty$ follows from Theorem 
\ref{theorem:simple-approximation} because of $\int_0^1 (1-t)H_X^2(t) dt < \infty$ which can bee seen by considering the trivial time-net $\{0,1\}.$
Consequently,
\[ \liminf_{N \to \infty}
   \sqrt{N} \left [ \inf_{\tau_N \in \Tau_N} a_S^{\rm opt}(F;\tau_N)\right ]
   \ge \frac{\vare}{\sqrt{2}}. \]
\end{proof}

Now we relate the approximation properties to the Besov regularity.
We recall that the nets $\tau_N^{\theta}$ were introduced in (\ref{eqn:nets})
and that for $\theta=1$ we obtain the equidistant nets.
\pagebreak

\begin{theorem}\label{thm:equidistant-besov}
For $\theta \in (0,1)$, $q\in [1,\infty]$, $Y\in \{ X,S \}$ and $F \in  \M$ the following assertions
are equivalent:
\begin{enumerate}[{\rm (i)}]
\item $F\in \B_{2,q}^\theta$.
\item $\Big \|(N^{\frac{ \theta}{2}  -\frac{1}{ q}}a_X^{\rm opt}(F; \tau^1_N) )_{N=1}^{\infty} \Big \|_{\ell_ q} < \infty$.
\end{enumerate}
If $\Omega = \{ S_t \not = 0\}$ for all $t\in [0,1]$, then {\rm (i)} and
{\rm (ii)} are equivalent to:
\begin{enumerate}
\item[{\rm(iii)}] $\Big \|(N^{\frac{ \theta}{2}  -\frac{1}{ q}}a_S^{\rm opt}(F; \tau^1_N) )_{N=1}^{\infty} \Big \|_{\ell_ q} < \infty$.
\item[{\rm(iv)}] $\Big \|(N^{\frac{ \theta}{2}  -\frac{1}{ q}}a_S^{\rm sim}(F; \tau^1_N,\Omega^N) )_{N=1}^{\infty} \Big \|_{\ell_ q} < \infty$.
\end{enumerate}
\end{theorem}
\bigskip
For the proof the following lemma is needed.

\begin{lemma}\label{lemma:comparison_H_X-H_S}
For $F\in \M$ and $t\in [0,1)$ one has that
\[ |H_S(t) - H_X(t)| 
   \le \mu(\rz) \| \vph_t\|_{L_2}. 
\]
Moreover,
\equa
 \les \Bigg | \kla \sum_{k=1}^N \int_{t_{k-1}}^{t_k} (t_k-t) H_S^2(t)dt \mer^\half&-&  \kla \sum_{k=1}^N \int_{t_{k-1}}^{t_k} (t_k-t) H_X^2(t) dt \mer^\half \Bigg |\\
&\le & \sqrt{\mu(\rz) |\tau|} \, \| F \|_{L_2}.
\tion
 \end{lemma}
\begin{proof}
From the definition we get that
\equa
      |H_S(t) - H_X(t)|
&\le& \kla \mu(\rz) \sum_{n=1}^\infty  n  n!  
               t^{n-1} \| h_{n-1} \|_{L_2(\mu^{\otimes n})}^2 \mer^\frac{1}{2} \\
& = & \kla \mu(\rz)^2 \sum_{n=1}^\infty    (n-1)!  
               t^{n-1} \|n h_{n-1} \|_{L_2(\mu^{\otimes (n-1)})}^2 \mer^\frac{1}{2} \\
& = & \mu(\rz) \| \vph_t\|_{L_2}.
\tion
Finally, 
\equa
 && \les  \left | \kla \sum_{k=1}^N \int_{t_{k-1}}^{t_k} (t_k-t) H_S^2(t)dt \mer^\half-  \kla \sum_{k=1}^N \int_{t_{k-1}}^{t_k} (t_k-t) H_X^2(t) dt \mer^\half \right |   \\
&\le& \kla \sum_{k=1}^N \int_{t_{k-1}}^{t_k} (t_k-t) |H_S(t) -H_X(t)|^2dt \mer^\half\\
&\le& |\tau|^\half |\mu(\rz)|^\half \kla \int_0^1 \| \vph_t\|_{L_2}^2 dt \, \mu(\rz) \mer^\half \\
&=& |\tau|^\half |\mu(\rz)|^\half \|F\|_{L_2}. 
\tion
\end{proof}
\medskip

\begin{proof}[Proof of Theorem \ref{thm:equidistant-besov}]
${\rm (i)}\Longleftrightarrow {\rm (ii)}$ follows from Theorem \ref{theorem:geiss-hujo-theta-q}
and Theorem \ref{theorem:simple-approximation} because
 \equal \label{H-curvature}
  H_X^2(t) = \frac{d^2}{d t^2} 
              \kla \sum_{n=1}^\infty \| I_n(f_n)\|_{L_2}^2 t^n \mer
   \sptext{1}{if}{1}
   F= \sum_{n=1}^\infty I_n(f_n).
\tionl
${\rm (iii)}\Longleftrightarrow {\rm (iv)}$ follows from Theorem \ref{theorem:opt-sim} and
${\rm (ii)}\Longleftrightarrow {\rm (iv)}$ from Theorem \ref{theorem:simple-approximation} and 
Lemma \ref{lemma:comparison_H_X-H_S}.
\end{proof}
\bigskip

\begin{theorem} \label{theorem:optimal-net-old} 
\begin{enumerate}[{\rm (a)}]
\item
For $F \in \M$ and $\theta\in (0,1]$ the following assertions are equivalent:
\begin{enumerate}[{\rm (i)}]
\item $F\in\B_{2,2}^{\theta}$.
\item $\sup_N N^\frac{1}{2} a_X^{\rm opt}(F;\tau_N^{\theta}) < \infty$.
\end{enumerate}
If $\Omega = \{ S_t \not = 0 \}$ for all $t\in [0,1]$, then {\rm (i)} and
{\rm (ii)} are equivalent to:
\begin{enumerate}
\item[{\rm(iii)}] 
     $\sup_N N^\frac{1}{2} a_S^{\rm opt}(F;\tau_N^{\theta}) < \infty$.
\item[{\rm(iv)}] 
     $\sup_N N^\frac{1}{2} a_S^{\rm sim}(F;\tau_N^{\theta},\Omega^N) < \infty$.    
\end{enumerate}
\item If the assertions  {\rm (i) - (ii)} hold, then we have
      \[    \lim_{N\to\infty}N \bet a_X^{\rm opt} (F;\tau_N^\theta)\rag^2  
         =  \frac{1}{2\theta} \int_0^1 (1-t)^{1-\theta} H_X^2(t) dt \]
     and if in addition  $\Omega = \{ S_t \not = 0\}$ for all $t\in [0,1]$, then
     \equa
           \lim_{N\to\infty}N \bet a_S^{\rm opt} (F;\tau_N^\theta)\rag^2  
     & = & \lim_{N\to\infty}N \bet a_S^{\rm sim} (F;\tau_N^\theta,\Omega^N)\rag^2 \\
     & = & \frac{1}{2\theta}  \int_0^1 (1-t)^{1-\theta} H_S^2(t) dt.
     \tion
\end{enumerate}

\end{theorem}
\medskip

\begin{proof} 
\underline{Part (a)}:
${\rm (i)} \iff {\rm (ii)}$ follows from Theorems \ref{theorem:geiss-hujo-optimal-net} and
\ref{theorem:simple-approximation} 
 because of \eqref{H-curvature}.

${\rm (ii)} \iff {\rm (iv)}$
From \cite[Lemma 3.8]{geiss-hujo} and Theorem \ref{theorem:simple-approximation}
it follows that the desired equivalence is equivalent to
\begin{equation}\label{eqn:H_finite}
   \int_0^1 (1-t)^{1-\theta} H_X^2(t) dt < \infty
   \sptext{.5}{if and only if}{.5}
   \int_0^1 (1-t)^{1-\theta} H_S^2(t) dt < \infty.
\end{equation}
In view of Lemma \ref{lemma:comparison_H_X-H_S} it is therefore sufficient to check
that
$\int_0^1 (1-t)^{1-\theta} \| \vph_t\|_{L_2}^2 dt < \infty$ which follows from
$\int_0^1 \| \vph_t\|_{L_2}^2 \mu(\rz) dt = \| F - \ew F\|_{L_2}^2 < \infty$.
\smallskip

${\rm (iv)} \iff {\rm (iii)}$ follows from Theorem \ref{theorem:opt-sim},
$a_X^{\rm opt}(F;\tau) \le \| F \|_{L_2}$ and $|\tau_N^\theta| \le 1/(\theta N)$
by (\ref{eqn:estimate_nets}).
\medskip

\underline{Part (b)}:
Let $\al(s):= 1-\kla 1-s\mer^{\frac{1}{\theta}}$ and 
$H:[0,1)\to [0,\infty)$ be non-decreasing and continuous such that 
$\int_0^1 (1-t)^{1-\theta}  H^2(t)dt < \infty.$ For any $\delta \in (0,1)$ 
and $\eta:= \alpha^{-1}(\delta)$ we observe that
\equa
      \frac{1}{2\theta}  \int_0^\delta (1-t)^{1-\theta}  H^2(t)dt 
& = & \half  \int_0^\delta \al'(\al^{-1}(t)) H^2(t)dt \\
& = & \half  \int_0^\eta \al'(s) \big [ H^2(\alpha(s)) \alpha'(s) \big ] ds.
\tion
Because
\[   \alpha'(s) 
   = \lim_{N\to\infty} \sum_{k=1}^N N \left [ \alpha\kla \frac{k}{N}   \wedge \eta \mer - 
                    \alpha\kla \frac{k-1}{N} \wedge \eta \mer   \right ]
                    \one_{\left [ \frac{k-1}{N},\frac{k}{N}\right )}(s)  \]
for $s\in [0,\eta)$ and all terms on the  right-hand side are bounded
by the Lipschitz constant of $\alpha$ on $[0,\eta]$, dominated convergence implies that
\equa
&   & \frac{1}{2\theta}  \int_0^\delta (1-t)^{1-\theta}  H^2(t)dt \\
& = & \lim_{N\to\infty} \half \sum_{k=1}^N
      \int_{\frac{k-1}{N}\wedge\eta}^{\frac{k}{N}\wedge\eta} N 
      \left [ \al\kla \frac{k}{N}\wedge \eta\mer -\al \kla \frac{k-1}{N} \wedge \eta\mer  
              \right ] \\
&   & \hspace*{17em}
      \left [ H^2(\al(s))\al'(s) \right ] ds \\
& = & \lim_{N\to\infty} N \sum_{k=1}^N H^2(t_{k-1}^{N,\theta}) 
      \frac{(t_k^{N,\theta}\wedge \delta-t_{k-1}^{N, \theta}\wedge \delta)^2}{2} \\
& = & \lim_{N\to\infty} N\sum_{k=1}^N  
      \int_{t_{k-1}^{N, \theta}\wedge \delta }^{t_k^{N,\theta}\wedge \delta}
            (t_k^{N,\theta}\wedge \delta   -t)H^2(t_{k-1}^{N,\theta})dt
\tion
where we use that $H$ is uniformly continuous on $[0,\delta]$.
From this we deduce that
\equa
& & \les \liminf_{N\to\infty} N\sum_{k=1}^N  \int_{t_{k-1}^{N,\theta}}^{t_k^{N, \theta}}(t_k^{N,\theta}-t)H^2(t)dt \nonumber \\
&\ge&  \liminf_{N\to\infty} N\sum_{k=1}^N  \int_{t_{k-1}^{N,\theta}\wedge \delta }^{t_k^{N, \theta}\wedge \delta}
            (t_k^{N,\theta}\wedge \delta   -t)H^2(t_{k-1}^{N,\theta})dt  \nonumber\\
& = & \frac{1}{2\theta}  \int_0^\delta (1-t)^{1-\theta}  H^2(t)dt \\
\tion 
for all $\delta\in (0,1)$ and therefore
\[ \liminf_{N\to\infty} N\sum_{k=1}^N  \int_{t_{k-1}^{N, \theta}}^{t_k^{N,\theta}}(t_k^{N, \theta}-t)H^2(t)dt
\ge  \frac{1}{2\theta}  \int_0^1 (1-t)^{1-\theta}  H^2(t)dt. \]
On the other hand,  \eqref{eqn:estimate_nets} implies
\equa 
  \int_\delta^1 \!\! N\sum_{k=1}^N \!\kla (t_k^{N,\theta}-t)\one_{\big [ t_{k-1}^{N,\theta},t_k^{N,\theta} \big )}(t)\mer \!\!H^2(t)dt  
   \le  \frac{1}{\theta}   \int_\delta^1 \!\!(1- t)^{1-\theta}  H^2(t)dt 
 \tion
 for  $  \delta  \in (0,1).$  Choose $\delta$ such that the right hand side is less than $\vare>0.$
 We conclude (also using the previous computations of part (b) and the uniform continuity of $H$ on $[0,\delta]$)
\equa 
&   & \les \limsup_{N\to\infty} N\sum_{k=1}^N  
      \int_{t_{k-1}^{N,\theta}}^{t_k^{N,\theta}}(t_k^{N,\theta}-t)H^2(t)dt  \\
&\le& \limsup_{N\to\infty} N\sum_{k=1}^N  
      \int_{t_{k-1}^{N,\theta}\wedge \delta}^{t_k^{N,\theta}\wedge \delta}
      (t_k^{N,\theta}-t)H^2(t)dt +\vare \\
& = & \lim_{N\to\infty} N\sum_{k=1}^N  
            \int_{t_{k-1}^{N,\theta}\wedge \delta}^{t_k^{N,\theta}\wedge \delta}
            (t_k^{N,\theta}\wedge \delta-t)H^2(t)dt +\vare \\
& = & \frac{1}{2\theta}  \int_0^\delta (1-t)^{1-\theta}  H^2(t)dt +\vare \\
&\le& \frac{1}{2\theta}  \int_0^1 (1-t)^{1-\theta}  H^2(t)dt +\vare
\tion
and
\[ \limsup_{N\to\infty} N\sum_{k=1}^N  \int_{t_{k-1}^{N,\theta}}^{t_k^{N,\theta}}(t_k^{N,\theta}-t)H^2(t)dt \le  \frac{1}{2\theta}  \int_0^1 (1-t)^{1-\theta}  H^2(t)dt.\]
Consequently,
 \[ \lim_{N\to\infty} N\sum_{k=1}^N  \int_{t_{k-1}^{N,\theta}}^{t_k^{N,\theta}}(t_k^{N,\theta}-t)H^2(t)dt = \frac{1}{2\theta}  \int_0^1 (1-t)^{1-\theta}  H^2(t)dt.\]
It follows from (\ref{eqn:H_finite}) that for $H\in \{ H_X,H_S\}$ our assumptions on
$H$ are satisfied. Hence 
 Theorem \ref{theorem:simple-approximation} implies the limit expressions for
$a_X^{\rm opt}$ and $a_S^{\rm sim}(\cdot;\cdot,\Omega^N)$  
(note that $c \to 1$ for $|\tau| \to 0$ in Theorem \ref{theorem:simple-approximation}). 
The relation for $a_S^{\rm opt}$ follows from that one for $a_S^{\rm sim}(\cdot;\cdot,\Omega^N)$,
Theorem \ref{theorem:opt-sim} and the fact that 
\[
      \lim_{N\to\infty} \sqrt{N} \sqrt{|\tau_N^\theta|} a_X^{\rm opt}(F;\tau_N^\theta) 
  \le \limsup_{N\to\infty} \sqrt{\frac{1}{\theta}}  a_X^{\rm opt}(F;\tau_N^\theta) 
   =  0 \]
where we have used (\ref{eqn:estimate_nets}) and, as in the proof of 
Theorem \ref{theorem:lower_bound}, the relation $\int_0^1(1-t) H_X^2(t)dt<\infty$ together with 
Theorem \ref{theorem:simple-approximation}.
\end{proof}

Using the results from \cite[Theorem 2.4]{sepp} one can derive from Theorem \ref{theorem:simple-approximation} for example the following assertion.
\medskip

\begin{corollary}
For $F \in \M $ one has the following equivalences:
\begin{enumerate}[{\rm (i)}]
 \item There is a constant $c>0$ such that 
\[ \inf_{\tau_N \in \Tau_N} a_X^{\rm opt}(F;\tau_N) \le \frac{c}{\sqrt{N}} \quad \text{ for } N=1,2,\ldots \text{ iff }
  \int_0^1H_X(t)dt <\infty.
\]
\item There is a constant $c>0$ such that 
\[ \inf_{\tau_N \in \Tau_N} a_S^{\rm sim}(F;\tau_N,\Om^N) \le \frac{c}{\sqrt{N}} \quad \text{ for } N=1,2,\ldots \text{ iff }
  \int_0^1H_S(t)dt <\infty.
\]
\end{enumerate}
\end{corollary}


\section{Examples}
\label{sec:examples}

\subsection{Preparations}
The following two lemmas provide information about the orthogonal
projection $\Pi:L_2\to \M \subseteq L_2$.
\medskip

\begin{lemma} \label{lemma:projection}
Given $G\in L_2$, $\theta\in (0,1)$ and $q\in [1,\infty]$,  one has that
\begin{enumerate}[{\rm (i)}]
\item $G\in \DD_{1,2}$ implies $\Pi(G) \in \DD_{1,2}$,
\item $G\in \B_{2,q}^\theta$ implies $\Pi(G) \in  \B_{2,q}^\theta$.
\end{enumerate}
\end{lemma}
\begin{proof}
The lemma follows from the fact that for
\[ G = \sum_{n=0}^\infty I_n (\alpha_n) \]
with symmetric $\alpha_n\in L_2^n$  the function $h_n$ from 
Definition \ref{definition:M} computes  as in  \eqref{hn-1} 
so that
$\| f_n \|_{L_2^n} \le  \| \alpha_n \|_{L_2^n}$
where $f_n$ is defined as in Definition \ref{definition:M}.
Hence, the statement can be derived (for example) from Theorem \ref{theorem:geiss-hujo-theta-q}
using  the monotonicity of $A$ with respect to $\|a_n\|_{E_n}$ and the definition of $\DD_{1,2}$.
\end{proof}
\medskip
\begin{lemma}\label{lemma:kgw-projection}
For a Borel function $f:\rz\to\rz$ with $f(X_1)\in L_2$ there are
symmetric $g_n\in L_2(\mu^{\otimes n})$ such that
\begin{equation}\label{eqn:representation-f(X_1)}
 f(X_1)=\ew f(X_1) + \sum_{n=1}^\infty I_n(g_n \one_{(0,1]}^{\otimes n}).
\end{equation}
Moreover, it holds that 
$\Pi (f(X_1)) = \sum_{n=1}^\infty I_n(f_n)$ with symmetric $f_n$ satisfying
\equal \label{eqn:h_g}
      f_n    ((t_1,x_1),...,(t_n,x_n)) 
& = & h_{n-1}(x_1,...,x_{n-1}) \nonumber\\
&:= & \int_\rz g_n(x_1,...,x_{n-1},x) \frac{\mu(dx)}{\mu(\rz)}
\tionl
on $0<t_1<\cdots<t_n <1$
and $\Pi(f(X_1))$ is the orthogonal projection of $f(X_1)$ onto
$I(X)$ defined in (\ref{eqn:I(X)}).
\end{lemma}

The representation (\ref{eqn:representation-f(X_1)}) is proved in
\cite{baumgartner:master} and \cite{baumgartner:geiss:2011} and is based 
on invariance properties of $f(X_1)$ that transfer to the chaos representation.
One could also use \cite[Section 6]{geiss-laukk}.
\bigskip

\begin{lemma}\label{lemma:phi_difference_quotient}
Let $f\in C^b_\infty(\rz)$ and $f(X_1)=\sum_{n=1}^\infty I_n(g_n \one_{(0,1]}^{\otimes n})\in \DD_{1,2}$
with symmetric $g_n\in L_2(\mu^{\otimes n})$. Then the martingale $(\vph_t)_{t\in [0,1)}$
given by (\ref{eqn:vphi}) and (\ref{eqn:h_g}) has a closure $\vph_1$, i.e.
$\ew (\vph_1|\ftn_t) = \vph_t$ a.s., with
\begin{multline*}
    \vph_1 
   = \int_\rz 
     \bigg [  \one_{\{x\not = 0\}} \frac{f(X_1+x)-f(X_1)}{x} 
          +     \one_{\{x     = 0\}} f'(X_1) \bigg ] 
            \frac{\mu(dx)}{\mu(\rz)} \pl a.s.
\end{multline*}
\end{lemma}
\medskip
 
\begin{proof}
From \cite[Proposition 5.1 and its proof]{geiss-laukk} it is known that
\begin{multline}\label{eqn:chaos_difference_quotient}
     \one_{\{ x\not = 0\}} \frac{f(X_1+x)-f(X_1)}{x} +  \one_{\{ x = 0\}}f'(X_1) \\
  =  \sum_{n=1}^\infty n I_{n-1} (g_n(\cdot,x)\one_{(0,1]}^{\otimes (n-1)}) 
 \pl \mu\otimes\mass \pl a.e.
\end{multline}
Consequently, (\ref{eqn:h_g}) implies that, a.s.,
\equa
&   & \int_\rz 
      \bigg [  \one_{\{x\not = 0 \}} \frac{f(X_1+x)-f(X_1)}{x} 
      + \one_{\{ x = 0 \}} f'(X_1) \bigg ] 
       \frac{\mu(dx)}{\mu(\rz)} \\
& = & \int_\rz 
      \left [ \sum_{n=1}^\infty n I_{n-1} \left (g_n(\cdot,x)\one_{(0,1]}^{\otimes (n-1)} \right )
               \right ] 
      \frac{\mu(dx)}{\mu(\rz)} \\
& = & \sum_{n=1}^\infty n I_{n-1} \left  (\int_\rz g_n(\cdot,x)  \frac{\mu(dx)}{\mu(\rz)}\one_{(0,1]}^{\otimes (n-1)}
      \right ) \\
& = & \sum_{n=1}^\infty n I_{n-1} \left (  h_{n-1} \one_{(0,1]}^{\otimes (n-1)}
      \right ) \\
& =: & \vph_1 
\tion
where the second equality follows by a standard Fubini argument.
\end{proof}
\begin{definition}\rm
For $\delta >0$ we let 
\[ \psi(\delta):= \sup_{\lambda\in\rz} \mass( |X_1 -\lambda|\le \delta ). \]
\end{definition}

\begin{example}
\label{example:small_ball_estimates} \rm
The small ball estimate
\begin{equation}\label{eqn:small_ball}
\psi(\delta) \le c \delta
\end{equation}
can be deduced if $X_1$ has a  bounded density.
As an example we use tempered $\alpha$-stable processes with 
$\alpha\in (0,2)$, given by the L\'evy measure 
\[ \nu_\alpha(dx) := \frac{d}{|x|^{1+\alpha}} (1+|x|)^{-m} \one_{\{x\not = 0\}} dx \]
with $d>0$ and $m\in (2-\al,\infty)$ being fixed parameters. Then 
\cite[Theorem 5]{sztonyk:2010-jtp} implies that $X_1$ has a bounded density.
\end{example}
\medskip

For $K\in \rz$ and $\vare\in (0,1]$ we let $f_{K,\vare}
\in C_b^\infty(\rz)$ with $f_{K,\vare}(x)=0$ if $x\le K$,
$f_{K,\vare}(x)=1$ if $x\ge K+\vare$, $0\le f_{K,\vare}(x) \le 1$ and  
$0\le  f'_{K,\vare}(x) \le 2/\vare$ for all $x\in \rz$.
\bigskip

\begin{lemma} \label{lemma:upper_bounds_difference_quotient_with_psi}
For $K\in \rz$ and $\vare>0$ we have that
\begin{multline*}
    \int_{\rz\setminus \{0\}} \ew \bet \frac{f_{K,\vare}(X_1+x)-f_{K,\vare}(X_1)}{x} \rag^2  \mu(dx) \\
\le 4 \frac{\psi(2\vare)}{\vare^2} \int_{0<|x|\le\vare} 
      x^2 \nu(dx) 
      + \int_{\vare < |x| < \infty} \psi(|x|)   \nu(dx).
\end{multline*}
\end{lemma}

\begin{proof}
We get that
\equa
&   &\les  \int_{\rz\setminus \{0\}} \ew \bet \frac{f_{K,\vare}(X_1+x)-f_{K,\vare}(X_1)}{x} \rag^2  \mu(dx) \\
& = & \ew \int_{0<|x| \le \vare} \bet \frac{f_{K,\vare}(X_1+x)-f_{K,\vare}(X_1)}{x} \rag^2 
      \mu(dx) \\
&   & + \ew \int_{\vare < |x| < \infty} \bet \frac{f_{K,\vare}(X_1+x)-f_{K,\vare}(X_1)}{x}
      \rag^2  \mu(dx) \\
&\le& \frac{4}{\vare^2} \mass( X_1 \in [K-\vare, K+2\vare]) \int_{0<|x|\le\vare} 
      x^2 \nu(dx) \\
&   & + \int_{\vare < x < \infty} 
      \mass ( X_1 \le K+\vare, X_1 + x \ge K) \nu(dx) \\
&   & + \int_{-\infty <x<- \vare}  \mass( X_1 + x \le K+\vare, X_1 \ge K)
       \nu(dx) \\
&\le& 4 \frac{\psi(2\vare)}{\vare^2} \int_{0<|x|\le\vare} 
      x^2 \nu(dx) \\
&   & + \int_{\vare < x < \infty}
      \mass ( |X_1-K| \le x) \nu(dx) \\
&   & + \int_{-\infty <x<- \vare}  \mass(K\le X_1 \le K-2x)
       \nu(dx) \\
&\le& 4 \frac{\psi(2\vare)}{\vare^2} \int_{0<|x|\le\vare} 
      x^2 \nu(dx) 
      +  \int_{\vare < |x| < \infty} \psi(|x|)   \nu(dx).
\tion
\end{proof}

\begin{lemma}\label{lemma:upper_bounds_difference_quotient_without_projection}
For $K\in \rz$ and $\vare>0$ the following assertions are true:

\begin{enumerate}[{\rm (i)}]
\item  \begin{minipage}[c]{4in}  \vspace*{.5em}
       \[     \int_{\rz\setminus \{0\}} \ew \bet \frac{f_{K,\vare}(X_1+x)-f_{K,\vare}(X_1)}{x} \rag^2  \mu(dx)
          \le \nu(\rz) \] \vspace*{1em} 
       \end{minipage}
\item If  $\psi(\delta)\le c \delta$, then 
      \begin{multline*}
          \int_{\rz\setminus \{0\}} \ew \bet \frac{f_{K,\vare}(X_1+x)-f_{K,\vare}(X_1)}{x} \rag^2 \mu(dx) \\
      \le  9 c \min \klae \frac{1}{\vare} \int_{\rz} x^2 \nu(dx),
                       \int_{\rz} |x| \nu(dx) \mere.
      \end{multline*}
\end{enumerate}             

\end{lemma}

\begin{proof}
\begin{enumerate}[{\rm (i)}]
\item Using $\mu(dx) = x^2 \nu(dx)$ on $\rz\setminus\{0\}$ one has that 
  \equa
   \int_{\rz\setminus \{0\}} \ew \bet \frac{f_{K,\vare}(X_1+x)-f_{K,\vare}(X_1)}{x} \rag^2  \mu(dx)
   \le  \nu(\rz).
  \tion

\item If $\psi(\delta)\le c \delta$, then we can bound the right-hand side in Lemma \ref{lemma:upper_bounds_difference_quotient_with_psi} by
      \equa
      &   & 4 \frac{\psi(2\vare)}{\vare^2} \int_{0<|x|\le\vare} x^2 d\nu(x) 
            +  \int_{\vare < |x| < \infty} \psi(|x|)   \nu(dx) \\
      &\le& \frac{8c}{\vare} \int_{\rz} x^2 d\nu(x) 
            +  c \int_{\vare < |x| < \infty} |x|   \nu(dx) \\
      &\le& \frac{8c}{\vare} \int_{\rz} x^2 \nu(dx) 
            + \frac{c}{\vare} \int_{\vare < |x| < \infty} x^2   \nu(dx) \\
      &\le& \frac{9c}{\vare} \int_{\rz} x^2 \nu(dx).
      \tion
      Moreover,
      \equa
      &   & 4 \frac{\psi(2\vare)}{\vare^2} \int_{0<|x|\le\vare} x^2 \nu(dx) 
            +  \int_{\vare < |x| < \infty} \psi(|x|)   \nu(dx) \\
      &\le&   8c \int_{0<|x|\le\vare} |x| \nu(dx)
            + c \int_{\vare < |x| < \infty} |x|   \nu(dx) \\
      &\le& 8 c \int_{\rz} |x| \nu(dx).
      \tion
\end{enumerate}
\end{proof}

\begin{lemma}\label{lemma:upper_bounds_difference_quotient_with_projection}
Let $f(x)=\chi_{[K,\infty)}(x)$ for some $K\in \rz$. 
Assume 
$\sigma=0$, 
$\int_\rz |x|^\frac{3}{2} \nu(dx)<\infty$ and assume 
that there is a $c>0$ such that
$\psi(\delta) \le c \delta$ for all $\delta>0$.
Then one has that
\[    \ew \bet
      \int_{\rz\setminus\{0\}}
      \left | \frac{f(X_1+x)-f(X_1)}{x} \right |  \mu(dx) \rag^2 
      \le \frac{c}{2} \left (\int_\rz |x|^\frac{3}{2} \nu(dx)\right )^2.
      \]
\end{lemma}
\medskip
   
\begin{proof}
For $d\nu_0 (x) := |x|^\frac{3}{2} \nu(dx)$ we get that
\equa
&   & \ew \bet
      \int_{\rz\setminus\{0\}}
      \left | \frac{f(X_1+x)-f(X_1)}{x} \right |  \mu(dx) \rag^2 \\
&\le& \ew \bet
      \int_{\rz}
      |f(X_1+x)-f(X_1)|  |x|^{-\frac{1}{2}} \nu_0(dx) \rag^2 \\
&\le& \nu_0(\rz)
      \ew 
      \int_{\rz}
      |f(X_1+x)-f(X_1)|^2  |x|^{-1}  \nu_0 (dx) \\   
&\le& \nu_0(\rz)
      \int_{\rz}
      \psi\kla\frac{|x|}{2} \mer |x|^{-1}  \nu_0 (dx) \\  
&\le& \frac{c}{2} \nu_0(\rz)^2.
\tion
\end{proof}

\subsection{Examples} Throughout 
the whole subsection we fix a real number $K$ and let
\[ f(x) := \one_{(K,\infty)}(x). \]
\underline{(a) Without projection on $\M$:}  We will obtain  the (fractional) smoothness
of $\one_{(K,\infty)}(X_1)$ in dependence of distributional properties of $X.$ Note that Lemma \ref{lemma:projection}
ensures that  $\Pi(\one_{(K,\infty)}(X_1))$  has  at least  the (fractional) smoothness of $\one_{(K,\infty)}(X_1).$
Our standing assumption, as mentioned in the beginning, is
$\int_{\rz} x^2 \nu(dx)<\infty$.
The case $C_1$ below confirms that for a compound Poisson process $X$ we have $\one_{(K,\infty)}(X_1) \in \DD_{1,2}.$
\bigskip

\begin{tabular}{|l|c|c|c|c|}\hline
 & $\sigma$  & $\psi$ & additional assumption on $\nu$ & {\small Smoothness}\\ \hline
$C_1$    & $\sigma=0$    & arbitrary                & $\int_{|x|\le 1} \nu(dx)<\infty$ & $\DD_{1,2}$ \\
$C_2$    & $\sigma=0$          & $\psi(\delta)\le c \delta$ &$\int_{|x|\le 1} |x|\nu(dx)<\infty$ & $\DD_{1,2}$ \\
$C_3$   & arbitrary               & $\psi(\delta)\le c \delta$ &    & $\B_{2,\infty}^\frac{1}{2}$ \\  \hline
\end{tabular}
\bigskip

To check this table assume that the chaos-decomposition of $f_{K,\vare}(X_1)$ is described by 
symmetric $g_n^{K,\vare}\in L_2(\mu^{\otimes n})$.
From (\ref{eqn:chaos_difference_quotient}) we derive in the case $\sigma=0$ that
\equa
      \sum_{n=1}^\infty n n! \| g_n^{K,\vare} \|_{ L_2(\mu^{\otimes n})}^2 
& = & 
      \sum_{n=1}^\infty n^2 \int_\rz   (n-1)! \| g_n^{K,\vare}(\cdot,x) \|_{  L_2(\mu^{\otimes (n-1)})}^2  
             \mu (dx) \\
& = & \sum_{n=1}^\infty n^2 \ew  \int_\rz I_{n-1} (g_n^{K,\vare}(\cdot,x)\one_{(0,1]}^{\otimes (n-1)})^2  \mu (dx) \\
& = & \int_\rz \ew \bet 
      \sum_{n=1}^\infty n I_{n-1} (g_n^{K,\vare}(\cdot,x)\one_{(0,1]}^{\otimes (n-1)}) \rag^2 \mu (dx) \\
& = & \int_{\rz\setminus \{0\}} \ew \bet 
      \frac{f_{K,\vare}(X_1+x)-f_{K,\vare}(X_1)}{x} \rag^2 \mu (dx)
\tion
so that
\[      \| f_{K,\vare}(X_1)\|_{\DD_{1,2}}^2 
    \le 1 + \int_{\rz\setminus \{0\}}  \ew \bet \frac{f_{K,\vare}(X_1+x)-f_{K,\vare}(X_1)}{x} \rag^2 \mu (dx). \]
\underline{Cases $C_1$ and $C_2$:} Exploiting Lemma \ref{lemma:upper_bounds_difference_quotient_without_projection}
gives that
\[ \sup_{m=1,2,...} \| f_{K,1/m}(X_1)\|_{\DD_{1,2}} < \infty. \]
Moreover
$\|  f_{K,1/m}(X_1) - \chi_{(K,\infty)}(X_1) \|_{L_2} \to_m 0$
by dominated convergence so that $C_1$ and $C_2$ follow by a standard argument.
\medskip

\underline{Case $C_3$:} As before we get from  \eqref{eqn:chaos_difference_quotient} that
\equa
 && \les  \| f_{K,\vare}(X_1)\|_{\DD_{1,2}}^2 \\
& \le \!\!& 1 + 
       \int_\rz \ew \bigg |  \one_{\{ x \not = 0\}} \frac{f_{K,\vare}(X_1+x)-f_{K,\vare}(X_1)}{x} + \one_{\{ x = 0\}}f_{K,\vare}'(X_1)
    \bigg |^2 \mu (dx).
\tion
Exploiting Lemma \ref{lemma:upper_bounds_difference_quotient_without_projection} and the 
property $0\le f'_{K,\vare}(x) \le 2/\vare$ we continue with
\equa
      \| f_{K,\vare}(X_1)\|_{\DD_{1,2}}^2 
&\le& 1 + \frac{9c}{\vare} \int_{\rz} x^2 d\nu(x)
       + \sigma^2 \frac{ 4}{\vare^2}  \psi \kla \frac{\vare}{2}\mer\\
&\le& 1 + \frac{9c}{\vare} \int_{\rz} x^2 d\nu(x)
       + \sigma^2 \frac{2c}{\vare}.
\tion
On the other hand,
\[    \| \chi_{(K,\infty)}(X_1) - f_{K,\vare}(X_1) \|_{L_2} 
  \le \sqrt{\psi \kla \frac{\vare}{2}\mer} 
  \le \sqrt{\frac{c\vare}{2}}. \]
Estimating the $K$-functional $K(u, \one_{(K,\infty)}(X_1);L_2,\DD_{1,2})$ by the help of the 
decomposition $ \one_{(K,\infty)}(X_1)=\Big [ \one_{(K,\infty)}(X_1)-  f_{K,\vare}(X_1)\Big]+ f_{K,\vare}(X_1) $
and optimizing over $\vare >0$ gives  $\chi_{(K,\infty)}(X_1)\in \B_{2,\infty}^\frac{1}{2}$.
\bigskip

\underline{(b) After projection on $\M$:} Here we have the following

\begin{proposition}\label{proposition:projection_D12} \rm
Assume that $\sigma=0$, $0<\int_\rz |x|^\frac{3}{2} \nu(dx)<\infty$ and that
$\psi(\delta) \le c \delta$. Then one has for all $K\in \rz$ that
\[ \Pi(\one_{(K,\infty)}(X_1)) \in \DD_{1,2}.\]
\end{proposition}

\begin{proof}
By the same reasoning as in the cases $C_1$ and $C_2$ it is sufficient to show that
\[ \sup_{m=1,2,...} \| \Pi(f_{K,1/m}(X_1))\|_{\DD_{1,2}} < \infty. \]
By (\ref{eqn:D12_boundedness_phi-martingale}) and Lemma \ref{lemma:phi_difference_quotient} it suffices to check that
\[ \sup_{m=1,2,...} \ew \bet\int_{\rz\setminus \{0\}} 
    \left [ \frac{f_{K,\frac{1}{m}}(X_1+x)-f_{K,\frac{1}{m}}(X_1)}{x}
  \right ] 
       d\mu(x) \rag^2 < \infty. \]
But this estimate follows from Lemma \ref{lemma:upper_bounds_difference_quotient_with_projection} and
the representation
\[ f_{K,\vare} (x) = \int_{-\infty}^x f_{K,\vare}'(y) d y 
                   = \int_\rz \one_{[y,\infty)}(x) f_{K,\vare}'(y) d y \]
and $\int_\rz f_{K,\vare}'(y) d y = 1$.
\end{proof}

\begin{example}\rm
An example for Proposition \ref{proposition:projection_D12} is obtained 
from Example \ref{example:small_ball_estimates}. Considering 
\[ \nu_\alpha(dx) = \frac{d}{|x|^{1+\alpha}} (1+|x|)^{-m} \one_{\{x\not = 0\}} dx \]
for $d>0$, $\alpha\in \kla 0,\frac{3}{2}\mer$ and $m\in (2-\al,\infty)$ gives $\psi(\delta)\le c \delta$ and
$0<\int_\rz |x|^\frac{3}{2} d\nu_\alpha(x)<\infty$, where $\al$ turns out to be the Blumenthal-Getoor index.
Using the results of \cite{laukkarinen} one can also show that
$\one_{(K,\infty)}(X_1) \not\in \DD_{1,2}$ for $\al\ge 1$ so that the projection 
$\Pi$ improves the smoothness of $\one_{(K,\infty)}(X_1)$ for $\al\in \left [ 1,\frac{3}{2} \mer$.
\end{example}

\begin{remark}\label{remark:broden-tankov}\rm

Using a Fourier transform approach  Brod\'en and Tankov \cite{broden-tankov} compute the discretization error under the historical 
measure for the  delta hedging as well as for a strategy which is optimal  under a given equivalent 
martingale measure. Using the equivalences of Theorem \ref{thm:equidistant-besov}  (i) $\iff $ (iv)    and Theorem \ref{theorem:optimal-net-old} (i) $\iff $ (iv)
one can also conclude  about the fractional smoothness of the projection of the considered digital option  from the computed  convergence rate for equidistant time nets.
\end{remark}


\newpage

\end{document}